\documentclass[12pt,english]{article}
\usepackage{babel}
\usepackage[cp1250]{inputenc}
\usepackage[T1]{fontenc}
\usepackage{amsfonts}
\usepackage{amsmath}
\usepackage{amsthm}
\usepackage{hyperref}
\usepackage{enumerate}
\usepackage{graphicx}
\usepackage{pictex}
\usepackage{color}

\newcommand{\rj}{r_{1}}
\newcommand{\rd}{r_{2}}
\newcommand{\kf}{\widetilde{K}}
\newcommand{\g}{\Gamma}
\newcommand{\gj}{\Gamma_{1}}
\newcommand{\gf}{\widetilde{\Gamma}}
\newcommand{\hd}{\hspace{0.2cm}}
\newcommand{\vd}{\vspace{0.2cm}}
\newcommand{\no}{\noindent}
\newcommand{\la}{\lambda}
\newcommand{\sk}{\overline{S}}
\newcommand{\skk}{\overline{{\sk}}}
\newcommand{\hj}{H^{1}(\Omega)}
\newcommand{\ld}{L^{2}(\Omega)}
\newcommand{\lap}{\Delta}
\newcommand{\rh}{\varrho}
\newcommand{\rhi}{\rh_{i}}

\newcommand{\rdti}{\rh_{i}^{-\frac{2}{3}}}

\newcommand{\coi}{\cos{\frac{2}{3}\theta_{i} } }
\newcommand{\te}{\theta}
\newcommand{\tei}{\te_{i}}
\newcommand{\pn}[1]{\frac{\partial #1}{\partial {\bf n}}}
\newcommand{\pna}[1]{\frac{\partial #1}{\partial n}{}_{|\g}}
\newcommand{\agf}[1]{{#1}_{|\gf}}

\newcommand{\eqq}[2]{\begin{equation}  #1  \label{#2} \end{equation}    }
\newtheorem{remark}{\textbf{Remark}}[section]
\newtheorem{corollary}{\textbf{Corollary}}[section]
\newtheorem{theorem}{\textbf{Theorem}}[section]
\newtheorem{proposition}{\textbf{Proposition}}[section]
\newcommand{\igj}{\int\limits_{\g_{1}}}
\newcommand{\igd}{\int\limits_{\g_{2}}}
\newcommand{\od}{\Omega_{\delta}}
\newcommand{\gd}{\Gamma^{\delta}}
\newcommand{\gdj}{\Gamma^{\delta}_{1}}
\newcommand{\ep}{\varepsilon}

\newcommand*{\re}{\mathop{\mathrm{Re}}\limits}
\newcommand{\ard}{(a,\rd)}
\newcommand{\arj}{(a,\rj)}
\newcommand{\rjb}{(\rj,b)}
\newcommand{\rr}{\mathbb{R}}

\newcommand{\ab}{\bar{\alpha}}
\newcommand{\bb}{\bar{\beta}}
\newcommand{\wa}{W_{\ab}}
\newcommand{\wam}{W_{-\ab}}
\newcommand{\wb}{W_{\bb}}
\newcommand{\wfa}{\widetilde{W}_{\ab}}
\newcommand{\wfam}{\widetilde{W}_{-\ab}}
\newcommand{\wfb}{\widetilde{W}_{\bb}}
\newcommand{\wfz}{\widetilde{W}_{0}}
\newcommand{\um}{U^{M}}
\newcommand{\umi}{U^{M}_{i}}
\newcommand{\cz}{c_{0}}
\newcommand{\dt}{\frac{2}{3}}
\newcommand{\wai}{W_{\ab,i}}
\newcommand{\wzi}{W_{0,i}}
\newcommand{\waj}{W_{\ab,1}}
\newcommand{\wac}{W_{\ab,4}}
\newcommand{\waim}{W_{-\ab,i}}
\newcommand{\wajm}{W_{-\ab,1}}
\newcommand{\wacm}{W_{-\ab,4}}
\newcommand{\wbi}{W_{\bb,i}}
\newcommand{\af}{\widetilde{\alpha}}
\newcommand{\bfa}{\widetilde{\beta}}

\newcommand{\bR}{\rr}

\newcommand{\m}[1]{\mbox{ #1}}

\newcommand{\po}{\pi/\omega}
\newcommand{\pot}{\te\pi /\omega}
\newcommand{\kpo}{k\pi/ \omega}
\newcommand{\kpot}{\te k \pi / \omega}

\newcommand{\fk}{\varphi_{k}}
\newcommand{\fkt}{\fk(\te)}

\def\bn{\bf n}
\def\cU{\mathcal{U}}
\def\kf{R_2}
\newcommand{\wf}{\widetilde{W}}
\newtheorem{definition}{\textbf{Definition}}[section]
\newtheorem{lemma}{\textbf{Lemma}}[section]

\begin{document}
\title{\bf Fine singularity analysis of solutions to the Laplace equation}
\author{Adam Kubica\\Faculty of Mathematics and Information Science\\
Warsaw University of Technology\\
ul. Koszykowa 75, 00-662 Warsaw, POLAND\\
{\tt A.Kubica@mini.pw.edu.pl}\\
Piotr Rybka\\Faculty of Mathematics, Informatics and  Mechanics\\
The University of Warsaw\\
ul. Banacha 2, 02-097 Warsaw, POLAND\\{\tt rybka@mimuw.edu.pl}}

\maketitle
\date{}

\abstract{We present here a fine  singularity analysis of  solutions to the
Laplace equation in special polygonal  domains in the plane. We assume
a piecewise constant Neumann data on one component of the boundary. Our
motivation is to study 
the so-called
Berg's effect \cite{berg}, \cite{gr-berg}.}


\section{Introduction}
We present here a fine  singularity analysis of  solutions to the
Laplace equation in special polygonal  domains in the plane. We assume
a piecewise constant Neumann data on one connected component of the boundary.

This topic is rather
well-studied so we have to
explain carefully the purpose of this research.
Here is our motivation, in \cite{gr-berg} the author claimed that the so-called
Berg's effect holds in the exterior of a straight circular cylinder in $\bR^3$.
Roughly speaking, this means that if $u$ is a harmonic function in the exterior
of a straight,
circular cylinder in $\bR^3$ with Neumann data constant on the bases
and the lateral surface, then its restriction to the boundary of the cylinder in question enjoys some monotonicity properties. We refer to \cite{gr-berg} for the exact formulation.
The point is that the statement arises from the observation made by Berg in the
1930's, see
\cite{berg}, that if one grows regular polyhedral crystals from the salt
solution in water, then the salt density restricted to the faces of the
crystal is an increasing function of the distance from the center of the
facet. The first attempt to explain this effect theoretically was done
by Seeger \cite{seeger}. However, until publication of \cite{gr-berg} no one
attempted to prove it in full generality.

However, P.G\'orka and A.Kubica pointed out that in \cite{gk} that the
original argument is
flawed. More precisely, the proof of \cite[Lemma 1.]{gr-berg} has a gap. This
Lemma
claims regularity of
solutions to the Laplace equation up to the boundary. Thus, the
question of validity of Berg's effect reopens.

Our ultimate goal is to settle the issue, but we will proceed in several
stages. The purpose of the present paper is to make the first step
toward understanding the problem in a two dimensional case. There is a separate
problem of behavior
of harmonic functions at infinity. So, in order to minimize unessential
difficulties we will consider a bounded domain only. Here, we consider the
following equation,
\begin{equation}\label{r1}
\left\{
\begin{array}{ll}
 \Delta u =0& \hbox{in }\Omega:= R_2\setminus \overline{R_1},\\
 u =0& \hbox{on } \partial R_2,\\
 \frac{\partial u}{\partial\bn} = u_n& \hbox{on } \partial R_1,
\end{array}
\right.
\end{equation}
We used here the following notation,
$R_1 =(-r_1,r_1)\times(-r_2,r_2)$ and $R_2 = \lambda_0 R_1$ with $\lambda_0>1$,
$\bn$ is the outer normal to $\Omega$ and
\begin{equation*}
 u_n =
\left\{
\begin{array}{ll}
 a & \hbox{for } |x_2| = r_2,\\
 b & \hbox{for } |x_1| = r_1.
\end{array} \right.
\end{equation*}
The question, which we are going to address, is:
What  conditions must $a$ and $b$ satisfy to guarantee that $u$
is singular? What are the conditions guaranteeing regularity of $u$?

Despite the effort of many people to study singularities of solutions to
elliptic
problems (see \cite{Grisvard}, \cite{dauge}, \cite{kondratiew}, \cite{kozlov},
\cite{nazarov}) such questions remain difficult. Partially,
this is due to the fact that the available tools are too
general. Namely, it is well known that if $u$ is a solution to (\ref{r1}),
then
\begin{equation}\label{r3}
u = v_r + c\phi,
\end{equation}
where $v_r$ is regular, i.e. $v_r\in H^2$, $\phi$ a singular, i.e. $\phi\in
H^1\setminus H^2$ and  $c$ is given by an integral  formula involving
boundary data $u_n$, see Lemma (\ref{str1}) for details.
For practical purposes it is very difficult to check if $c$ vanishes. Here are
our
results, where we address a planar bounded domain.

\begin{theorem}\label{tm2}{\sl(a rectangle inside a scaled rectangle) Let us suppose
that $R_1$ is a general rectangle as described earlier. There are  unique
numbers $\alpha_{1}, \beta_{1}$ related with $\Omega$ such that $|\alpha_{1}|+
|\beta_{1}|>0$ and if $u$ is
a weak solution to (\ref{r1})}, then
\[
u \in C^{1}(\overline{\Omega}) \iff a\alpha_{1}+b\beta_{1}=0.
\]
\end{theorem}

Once we established Theorem \ref{tm2} for a generic rectangle we may turn to
a special case of a square.

\begin{theorem}\label{tm1}{\sl (a square inside a scaled square) Let us suppose that
$R_1$ 
is a square $R_1=Q=(-R,R)^2$. If
$u$ is a weak solution to (\ref{r1}), then }
\[
u \in C^{1}(\overline{\Omega}) \iff a=b,
\]
i.e. number $\alpha_{1}, \alpha_{2}$ from theorem~\ref{tm2} satisfy $\alpha_{1}= -\alpha_{2}\not =0$.
\end{theorem}

At the technical level our results for bounded domains in the plain
are proved by  a very careful analysis of behavior of regular level
sets of harmonic functions in $\Omega\subset \bR^2$. The boundary of
$\Omega$ has exactly
two connected components, which are polygons. We will call by
$\Gamma$ the inner part and $\gf$ the outer part of the boundary. In
principle, the description of the singularities is well-known,
see the fundamental monograph \cite{Grisvard}. However, this description
is not effective.

On a more fundamental level, our paper does not make Berg's effect invalid.
It suggests that it is a rather rare phenomenon, which could be observed when
crystals are near equilibrium with the environment.  The above result strongly
suggests that contrary to the claim
made in \cite{gr-berg}  solutions to \cite[eq. (2.2)]{gr-berg} in
general are singular. But influence of
the singularity on Berg's effect will be studied elsewhere.

\section{Preliminaries}

We first present facts on corner singularities of harmonic functions, then we will look at their level sets.

\subsection{On singular solutions to Laplace equation}

We introduce here the necessary notions and background material from
\cite{Grisvard}. We begin with the definition of the domain $\Omega$. First, we
set  $R_1=(-\rj,\rj)\times (-\rd, \rd)$, which will be the inner
rectangle. We take any $\lambda_{0}>1$ and we set
$\kf=\lambda_{0} R_1$. The domain of our harmonic functions is
$$
\Omega=\kf\setminus \overline{R_1} \equiv
\lambda_0 R_1 \setminus \overline{R_1}.
$$
The boundary of $\Omega$ consists of two connected components. For our purposes
we will break it down even further. We shall write,
\begin{eqnarray*}
 &&\g
=\partial R_1=\g_{1} \cup\g_{2} \cup\g_{3} \cup\g_{4}\cup S_{1} \cup
S_{2} \cup S_{3} \cup S_{4},\\
&&\gf =\partial \kf=\gf_{1}
\cup\gf_{2} \cup\gf_{3} \cup\gf_{4}\cup \widetilde{S}_{1} \cup
\widetilde{S}_{2} \cup \widetilde{S}_{3} \cup \widetilde{S}_{4} ,
\end{eqnarray*}
where $\g_{i}$, $\gf_{i}$  are sides of rectangles and
$S_{i}$, $\widetilde{S}_{i}$ are their
vertices, $i=1,\ldots,4$. To be precise, we set $\g_{1}=\{(t,\rd): \hd t\in
(-\rj,\rj) \}$, with the respective definition of $\widetilde\Gamma_1$ and
$\Gamma_2$, $\Gamma_3$, $\Gamma_4$ are the remaining sides of $R_1$ visited
counterclockwise. $\widetilde\Gamma_j$, $j=2,3,4$ are respectively defined for
$\kf$. We also set $S_{i}=\overline{\g_{i}} \cap \overline{\g_{i+1}}$, with the understanding
that $\Gamma_{4+1} = \Gamma_1$ and in the same manner we define
$\widetilde S_j$. The distance from vertex $S_i$ is $\rh_{i}$. We also set
$\rh=\min\limits_{i=1,...,4}{\rh_{i}}$.

For $i=2,4$, we set $\te_{i}$ to be the angle measured at $S_i$
from $\g_{i}$ to $\g_{i+1}$. At the same time for $i=1,3$, we set
$\te_{i}$ to be the angle measured from  $\g_{i+1}$ to $\g_{i}$. We
denote by $\eta_{i}=\eta_{i}(\rh_{i})$ a cutoff function equal  $1$ in a
neighborhood of $S_{i}$ with support in $B(S_{i},\min\{r_{1},r_{2}\} )$.
Furthermore, let
$\psi_{i}$ be a cutoff function equal to $1$ in the neighborhood of $\g_{i}$.

Before plunging into analysis of our problem, we state a more basic result.
\begin{proposition}\label{podst}
 Let us suppose that  $\omega \in (\pi, 2\pi)$, then we set $U=
\{(x,y)\in\mathbb{R}^{2}: \hd r\in (0,r_{0}), \hd \te \in (0, \omega) \}$, where $(r,\te)$ are polar coordinate in $\rr^{2}$. We
assume that $S \in L^{2}(U)$ is a solution to the following problem,
\eqq{ \left\{
\begin{array}{rcl}
\Delta S =0 & \m{ in } & U, \\
\frac{\partial S}{\partial n }=0  & \m{ for  } & \te =0,  \omega \\
\end{array}
\right. }{a-a}
Then, there exist constants  $c_{k,m}$ such that we have
\eqq{
S= c_{1,1}\sqrt{2/\omega } r^{-\po} \cos{\pot}+ c_{1,0} 1/\sqrt{\omega}\ln{r} +
c_{2,0}/\sqrt{\omega}+ \sum_{k=1}^{\infty} c_{2,k} r^{\kpo} \cos{\kpot} .
}{aa-a}
Moreover, for $l=0,1,\dots$ we have
\eqq{\sum_{k \geq \frac{\omega l }{\pi}}^{\infty} c_{2, k } r^{\kpo} \cos{\kpot}
\in C^{l}(\overline{U}).  }{b-a}
\end{proposition}

\begin{remark}
For  $\omega \in (0,\pi)$ the statement is changed by dropping the first term in right hand side of (\ref{aa-a}) and in (\ref{b-a}) we get $C^{l+1}(\overline{U})$.
\end{remark}

\begin{proof}
Function $S$ is smooth up to the boundary away from the origin, because it is
harmonic inside $U$ and can be harmonically continued across the boundary by
even reflection. Thus, we have to establish its asymptotic behavior near origin.
Without the loss of generality we may assume that $\| S \|_{L^{2}(U)}=1$. Then
we set $\fkt= \sqrt{\frac{2}{\omega}} \cos{\kpot} $ for $k=1,2, \dots $ and
$\varphi_{0}(\te)= \frac{1}{\sqrt{\omega}}$. Functions $\{ \fk
\}_{k=0}^{\infty}$ form an orthonormal basis of $L^{2}(0, \omega)$, thus
\eqq{
S(r, \te) = \sum_{k=0}^{\infty} w_{k}(r) \fkt,
}{ab-a}
for all $r \in (0,r_{0})$. At the same time this  series converges in  $L^{2}(0 ,
\omega)$. Its coefficients are given by the following formula,
\[
w_{k}(r)= \int_{0}^{\omega} S(r, \te) \fkt d\te.
\]
From (\ref{a-a}) we obtain an ODE for $w_{k}$:
\[
w_{k}''+\frac{w_{k}'}{r}- (\kpo)^{2}\frac{w_{k}}{r^{2}}=0, \hd k=0,1, \dots,
\]
in other words
\eqq{
w_{0}(r)= c_{1,0}\ln{r}  + c_{2,0}, \hd \hd
w_{k}(r)= c_{1, k }r^{-\kpo} + c_{2,k} r^{\kpo} \hd \m{ for }  k=1,2,\dots.
}{ac-a}
We shall show that
\eqq{c_{1,k}=0 \m{ for } k>1.}{d-a}
For this purpose we notice that
$\int_{0}^{r_{0}}
|w_{k}(r)|^{2}rdr= \int_{0}^{r_{0}} |\int_{0}^{\omega} S(r, \te ) \fkt d \te
|^{2} rdr \leq \| S \|^{2}_{L^{2}(U)}=1$. On the other hand
\[
\int_{0}^{r_{0}} |c_{1, k } r^{- \kpo}|^{2}rdr = |c_{1,k } |^{2} \frac{r^{2(1
-\kpo)}}{2(1- \kpo)} \Big|^{r_{0}}_{0}.
\]
This integral is finite, hence $c_{1, k }=0$ or $1- \kpo>0$, which
implies (\ref{d-a}). Thus, from (\ref{ab-a})-(\ref{d-a}), we infer
(\ref{aa-a}).

In order to see  (\ref{b-a}), we need estimates on coefficients
$c_{2,k}$ for $k>1$. For this purpose, we fix  $\delta \in (0,r_{0})$ and we set
$a_{k}\equiv w_{k}(\delta)= \int_{0}^{\omega} S(\delta, \te) \fkt d \te $. Then,
it is easy to see that $|a_{k}| \leq C(\delta)$, because $S$ is smooth away from
the origin. Then,  $w_{k}(\delta) = c_{2, k }\delta^{\kpo}$. Hence,
\eqq{|c_{2,k }|\leq C(\delta)\delta^{- \kpo}. }{ad-a}
In this way for $k \geq \omega l /\pi$, we obtain
\[
\| D^{l} (c_{2, k } r^{\kpo} \cos{\kpot} ) \|_{C(U)} \leq C(l)  k^{l}
|c_{2, k}|.
\]
We infer from (\ref{ad-a}) that the series $\sum_{k \geq \frac{\omega l
}{\pi}}^{\infty} c_{2, k } r^{\kpo} \cos{\kpot}$ converges in
$C^{l}(\overline{U}) $.
\end{proof}

We shall introduce a couple of 
functions, which are necessary in the description of singularities of solutions
to
(\ref{r1}). The first one is $\skk$.

\begin{definition}\label{def1}{\rm
 (Very weak solution $\skk$). Let $w\in\hj$ be a week
solution to the following  problem,
\[
\left\{
\begin{array}{ll}
\lap w = - \lap ( \sum\limits_{i=1}^{4}\eta_{i}
\varrho_{i}^{-\frac{2}{3}} \coi)& \hbox{in }\Omega,\\
\pna{w}=0 \qquad\qquad\hbox{and}&
\agf{w}=0. \\
\end{array}
\right.
\]
We notice that
$\lap ( \sum\limits_{i=1}^{4}\eta_{i}\varrho_{i}^{-\frac{2}{3}} \coi)\in \ld$.
We set
\begin{equation}\label{pr1}
\widetilde{{\widetilde{S}}}=w+\sum\limits_{i=1}^{4}\eta_{i}
\varrho_{i}^{-\frac{2}{3}} \coi ,
\end{equation}
hence, $0\not \equiv
\widetilde{{\widetilde{S}}} \in \ld$. We finally define }
\begin{equation}\label{pr1+1}
 \skk=\widetilde{{\widetilde{S}}}/ \| \widetilde{{\widetilde{S}}} \|_{\ld} \equiv c_0 \widetilde{{\widetilde{S}}}.
\end{equation}
\end{definition}

The basic properties of  $\skk$ are stated below.

\begin{corollary}\label{uwj} {\sl
 Function $\skk$, given by the above formula, is the only one, (up to the sign),
with the following properties,}
\eqq{ \left\{
\begin{array}{l}
\lap \skk = 0, \\
\pna{\skk}=0, \\
\agf{\skk}=0, \\
\end{array}
\right. }{d}
\eqq{ \| \skk \|_{\ld}=1,}{ddod}
\eqq{\skk(x,y)=\skk(-x,y)=\skk(x,-y)=\skk(-x,-y), }{e}
\eqq{ \skk \in \ld \setminus \hj.}{f}
\end{corollary}

\begin{proof}
We claim that $\skk$ and  $-\skk$ are the only functions
satisfying
(\ref{d})-(\ref{f}). Indeed, from  \cite[Theorem~2]{K1} and
\cite[Corollary~6]{K}, we deduce that, $\mathcal{V}$, the space of functions
which
satisfy (\ref{d}) and (\ref{f}) is spanned by four linearly independent
functions. Each of them corresponds to  one non convex corner of $\Omega$.
Symmetries (\ref{e}) reduce the dimension of $\mathcal{V}$ to one and from
(\ref{ddod}) we get the claim.
\end{proof}

\vd

We define the second important function.

\begin{definition}\label{def2} {\rm (singular solution  $\sk$).
 Let  us suppose that $ \skk\in\ld $ is given by Definition \ref{def1}. Then
$\sk\in\hj$ is  a unique weak solution to the following equation,}
\[
\left\{
\begin{array}{ll}
\lap \sk = \skk & \hbox{\rm in }\Omega, \\
\pna{\sk}=0, &
\agf{\sk}=0. \\
\end{array}
\right.
\]
\end{definition}

Having functions $\sk$ and $\skk$ at hand, we can provide a description of
singular solutions to (\ref{r1}). In order to do this we introduce an
auxiliary function
\[
f=-a (y-\rd)\psi_{1} +b (x+\rj)\psi_{2} +a (y+\rd)\psi_{3}
-b (x-\rj)\psi_{4}\in C^{\infty}(\Omega),
\]
where $\psi_{i}$ are cut off functions equal to one on some neighborhood of $\g_{i}$ and vanishing \m{on $\gf$.}

\begin{lemma}\label{str1}
 Let us suppose that $u \in \hj$ is a unique weak solution to (\ref{r1}). Then  $u$  has the following form,
\[
u=
u_{r}+(c_{a}+c_{b})\sk,\qquad u_{r}\in
C^{1}(\overline{\Omega}),
\]
where $c_{a}+ c_{b}=-\int\limits_{\Omega} \skk \lap f$.
\end{lemma}

\begin{proof}  Such a decomposition is a general fact, see
\cite[Theorem 1]{K1}. Now, the point is to calculate $c_{a}+c_{b}$. Obviously, $f$ satisfies boundary conditions (\ref{r1}${}_{2,3}$). Now, let
$v\in \hj$ be a weak solution to the problem
\[
\left\{
\begin{array}{ll}
\lap v = -\lap f &\mbox{in }\Omega,\\
\pna{v}=0, &
\agf{v}=0. \\
\end{array}
\right.
\]
According to \cite[Theorem 1]{K1} and its proof $v=
v_{r}+c\sk$, where $v_{r}\in H^{2}(\Omega)$  and
$\int_\Omega \Delta v_r \skk =0,$
where $\skk$ satisfies (\ref{d})-(\ref{f}). Then
it is easy to see that
$c=-\int\limits_{\Omega} \skk \lap
f$.

\no From the uniqueness of weak solutions we get $u=v+f$, so $u=v_{r}+f+c\sk$,
where $v_{r}+f \in H^{2}(\Omega)$. From Proposition~\ref{podst} we deduce that
$v_{r}+f \in C^{1}(\overline{\Omega})$. Finally, we see that $ c_{a}+
c_{b}=-\int\limits_{\Omega} \skk \lap f$, hence  the proof is finished.

\end{proof}

We shall see that despite a seemingly arbitrary choice of $f$, the definition of
$c$ is universal.
\begin{proposition}{\rm Let us suppose that $f$ is given above and $\skk$ is as
in Definition \ref{def1}. Then,}
\eqq{ \int\limits_{\Omega} \skk \lap f = 2a \igj \skk+2b\igd \skk.
}{c}
\end{proposition}

\begin{proof} The argument will be split in a number of steps.

\no {\it Step 1.} Let $\od=\kf\setminus
\overline{(-\rj-\delta,\rj+\delta)\times(-\rd-\delta,\rd+\delta)}$. The
regularity of $f$ and the boundedness of $\Omega$ imply that $ \skk \lap f \in
L^{1}(\Omega)$. Thus
\[
\lim_{\delta
\rightarrow 0 }\int\limits_{\od} \skk \lap f =\int\limits_{\Omega}
\skk \lap f.
\]
At the same time, $\skk$ is harmonic 
in $\od$, so we have
\[
\int\limits_{\od} \skk \lap f=
\int\limits_{\partial \od} \skk
\pn{f}-
\int\limits_{\partial \od} \pn{\skk} f.
\]
We split the boundary of
$(-\rj-\delta,\rj+\delta)\times(-\rd-\delta,\rd+\delta)$ exactly in the same
way as we did it earlier, so that we shall write $\partial \od =
\overline{\gd_{1}}\cup\overline{\gd_{2}}\cup\overline{\gd_{3}}\cup\overline{\gd_{4}}\cup \gf$.

\no{\it Step 2.} We will prove that
\eqq{\lim\limits_{\delta \rightarrow 0 } \int\limits_{\gdj } \skk
\pn{f}=\int\limits_{\gj } \skk \pn{f}=a\int\limits_{\gj } \skk.}{b}
From Proposition~\ref{podst} we deduce that  $|\skk|\leq
c_{0}\rh^{-\frac{2}{3}}$. Since by definition $ |\pn{f}|\le C$,
then we also have $|\skk\pn{f}|\leq c\rh^{-\frac{2}{3}}$, thus its integral over
any segment of the length $b$ is smaller than $6cb^{1/3}$. Therefore for  any
$\ep>0$  there exists $\ep_{1}>0$, such that for any  $\delta\in (0, \ep_{1})$
the following
estimate
\[
\Big|
\int\limits_{\gdj\setminus ([-\rj+\ep_{1},\rj-\ep_{1}]\times\{r_2+\delta\})}
\skk \pn{f}
\Big|+\Big|
\int\limits_{\gj\setminus ([-\rj+\ep_{1},\rj-\ep_{1}]\times\{r_2\}}
\skk \pn{f}
\Big| \leq
\frac{\ep}{2}.
\]
holds. Moreover, for fixed $\ep_{1}$, we have
\[
\skk(x, r_2+\delta)\rightarrow \skk(x, r_2),\qquad
x\in[-\rj+\ep_{1},\rj-\ep_{1}],
\]
as  $\delta$ converges to zero and the convergence is uniform, because $\skk$ is
smooth away from vertices. Then
\begin{eqnarray*}
\Big| \int\limits_{\gdj} \skk \pn{f} - \int\limits_{\gj} \skk \pn{f}
\Big|&\leq& \Big| \int\limits_{\gdj\setminus
([-\rj+\ep_{1},\rj-\ep_{1}]\times\{r_2+\delta\})} \skk \pn{f} \Big|+\Big|
\int\limits_{\gj\setminus ([-\rj+\ep_{1},\rj-\ep_{1}]\times\{r_2\})} \skk
\pn{f}
\Big|
\\
&+&\Big|
\int\limits_{\gdj \cap ([-\rj+\ep_{1},\rj-\ep_{1}]\times\{r_2+\delta\})}
\skk\pn{f} -
\int\limits_{\gj\cap  ([-\rj+\ep_{1},\rj-\ep_{1}]\times\{r_2\})} \skk
\pn{f} \Big| \rightarrow 0,
\end{eqnarray*}
when $\delta$ goes to  $0$, as a result (\ref{b}) holds.

The remaining cases of $\Gamma_{i}$ for $i=2,3,4$ are handled with in the same
way.

\no {\it Step 3.}
We claim that
\[
\lim_{\delta\rightarrow 0}\int\limits_{\Gamma^{\delta}_{i}} \pn{\skk} f  = 0.
\]
First, we will  notice that
 \[
\lim\limits_{\delta \rightarrow 0 }
\int\limits_{\Gamma^{\delta}_{i}} \pn{\skk} f  =
\int\limits_{\Gamma_{i}} \pn{\skk} f ,\qquad i=1,\ldots,4.
\]
By the definition of $f$, we get $|f|\leq c_{0}\rh$. On the other hand, using
Proposition~\ref{podst} we get $|\pn{\skk}|\leq c_{0} \rh^{-\frac{5}{3}}$,
hence $|\pn{\skk} f|\le c\rh^{-\frac{2}{3}}$. Therefore we may proceed as in
Step 2 and calculate the above limit. Finally, we see that $\pn{\skk}$ vanishes
on $\Gamma$, hence the
claim follows.

\no {\it Step 4.}
Integrals over
 $\gf$ vanish, because the support of $f$ does not intersect  $\gf$.
Considering the boundary values of  $\pn{f}$, we infer
(\ref{c}).
\end{proof}

\begin{corollary}
Let us suppose that $u \in \hj$ is a unique weak solution to (\ref{r1}). Then
\eqq{
u\in C^{1}(\overline{\Omega}) \iff a \igj \skk+b\igd \skk=0.
}{chgl}
\label{wnreg}
\end{corollary}

\begin{proof}
If $u\in C^{1}(\overline{\Omega}) $ is the weak solution of (\ref{r1}), then
necessarily $c_{a}+ c_{b}=0$, because $\sk \not \in C^{1}(\overline{\Omega})$.
Then, from (\ref{c}) we
get $a \igj \skk+b\igd \skk=0$. The other implication is obvious.
\end{proof}

\begin{remark}
We see that the issue of regularity of solution of problem (\ref{r1}) is reduced
to calculating integrals $\igj \skk$ and $ \igd \skk$. However, we can not do it
directly. This is the main obstacle, related to that function $\skk$ is not
given explicitly.
In our further analysis we concentrate only on these integrals. More precisely,
we will show that at least one function of 
$\skk_{|\g_{1}}$,
$\skk_{|\g_{2}}$ is positive or negative. Then, at least one integral $\igj
\skk$ or $ \igd \skk$ is non zero. This means that the set of $(a,b)\in
\rr^{2}$, for which the solution of (\ref{r1}) is regular, is just a straight
line.
\end{remark}

\subsection{Very weak solutions}

\begin{lemma}\label{step2}
There is  $\cU$, a neighborhood of vertices $S_{i}$ such that $\nabla \skk (x,y)\ne 0$ in $\cU$.
\end{lemma}
\begin{proof} In order to see this we recall the form of $\skk$, see (\ref{pr1}) and (\ref{pr1+1}). We
notice that in a sufficiently small neighborhood of vertices $S_{i}$ the term
 $\nabla (\rh^{-\frac{2}{3}}_{i} \coi) $ dominates $\nabla w$.
More precisely, from Proposition~\ref{podst}, we deduce that
$w=\sum\limits_{i=1}^{4}c\eta_{i}\rhi^{\dt}\coi+h$, where $h$ belongs to
 $C^{1}(\overline{\Omega})$. Therefore,  we conclude that
$\rh^{\frac{1}{3}}|\nabla
w|\leq c_{1}$, while $|\nabla (\rh^{-\frac{2}{3}}_{i}
\coi)|=\frac{2}{3}\rh_{i}^{-\frac{5}{3}}$. Then, by the triangle inequality
we have $|\nabla(\rh^{-\frac{2}{3}}_{i} \coi)|-|\nabla w|\leq
|\nabla \widetilde{\widetilde{S}}|$, as a result we see,
$\rhi^{\frac{1}{3}}|\nabla(\rh^{-\frac{2}{3}}_{i}
\coi)|-\rhi^{\frac{1}{3}}|\nabla w|\leq \rhi^{\frac{1}{3}}|\nabla
\widetilde{\widetilde{S}}|$. This implies that
 $\frac{2}{3}\rh_{i}^{-\frac{4}{3}}-c_{1}\leq \rhi^{\frac{1}{3}}|\nabla
\widetilde{\widetilde{S}}|$, and then
$\frac{2}{3}\rhi^{-\frac{1}{3}}(\rh_{i}^{-\frac{4}{3}}-c_{1})\leq
|\nabla \widetilde{\widetilde{S}}|$,  which means that for sufficiently small $\rh_{i}$
we have $\nabla \skk \ne 0$.
\end{proof}

\begin{lemma}\label{step6}
For each $k>0$ there is $M>0$ such that, if we define $\um$ by
$$
\um= \bigcup\limits_{i=1}^{4}\umi,
$$
where
\begin{equation}\label{rn27pol}
 \umi=\{(x,y)\in \Omega:\hd \rh_{i} \leq
M^{-\frac{3}{2}}|\coi|^{\frac{3}{2}} \},
\end{equation}
then
\eqq{|\skk|_{|\um}>k.}{g}
\end{lemma}
\begin{proof}
Let us fix $k>0$. We set $\umi$ by formula (\ref{rn27pol}).
Obviously, function $|\rdti \coi|$ is bounded below by
$M$ in $\umi$. Since $w$ in the definition of $\skk$ is continuous in
$\overline{\Omega}$, then the number
$m=\max\limits_{\Omega}{|w|}$ is well-defined.
We recall the shorthand $c_{0}=\|
\widetilde{\widetilde{S}} \|_{\ld}^{-1}$. By the definition of $U^{M}_{i}$, on
the set $U^{M}_{i}\cap \{\eta_{i}(\rh_{i})=1 \}$ we have
\begin{eqnarray*}
c_{0} M&\leq &\big| c_{0}\rdti \coi \big|=\big|\skk
- c_{0}w\big| \leq  |\skk| + c_{0}m.\\
\end{eqnarray*}
In other words,
\[
c_{0}M-c_{0}m \leq |\skk| \mbox{ \hd on \hd } \umi,
\]
for $M$ so large that $U^{M}_{i}\subseteq \{\eta_{i}(\rh_{i})=1
\}$. Finally, choosing  a constant $M$ so big that the left-hand-side of the
above inequality is bigger than $k$, we get (\ref{g}).
\end{proof}


\begin{lemma}\label{h} 
For each $k>0$ there exists $\ep_{2}>0$  such that for  all  $\ep\in (0,\ep_{2})$ function
$\skk$ restricted to $\Omega\cap\partial B(S_{i}, \ep)\setminus \umi$ is
strictly decreasing with respect to the angle $\te_{i}$ for all
$i=1,\dots, 4.$
\end{lemma}
\no The constant $M$ and sets $\umi$ are given by Lemma~\ref{step6}.
\begin{proof}
Actually, in a neighborhood of $S_{i}$ we have $\skk
= c_{0}\rdti \coi +c_{0}w $. Let us recall that Proposition~\ref{podst} implies
that
$\rh^{\frac{1}{3}}|\nabla w|$ is bounded in
$\overline{\Omega}$, hence the number $m_{2}=
\max\limits_{\overline{\Omega}}{\rh^{\frac{1}{3}}|\nabla w | }$ is well-defined.
We obviously have
\[
\rhi^{-\frac{2}{3}}\frac{\partial }{\partial \tei}\skk= -\frac{2}{3}
\cz \rhi^{-\frac{4}{3}} \sin{\dt \tei} + \cz
\rhi^{-\dt}\frac{\partial }{\partial \tei}w,
\]
where $\Big|\rhi^{-\dt}\frac{\partial }{\partial \tei}w\Big|
\leq \rh^{\frac{1}{3}} |\nabla w | \leq m_{2}$.  We recall  that according to
the definition of $\umi$, we have
$M\geq \rdti|\coi |$ in $\Omega\setminus \umi$. As a result, $\frac{2}{3}\tei
\in
(\arccos(M\rhi^{\dt}),\arccos(-M\rhi^{\dt}))$. Then, on this interval, we have
$\inf{\sin{\dt \tei}}=
\sin(\arccos(M\rhi^{\dt}))>0$. This implies that  in a neighborhood of
$S_{i}$ for points not belonging to  $\umi$, we have
\[
-\frac{\partial }{\partial \tei}\skk \geq \dt \cz \rdti
\sin(\arccos(M\rhi^{\dt})) - \rhi^{\dt} \cz m_{2}.
\]
Certainly, the right-hand side of the above inequality monotonically grows to
 $\infty$, when $\rhi $ tends to zero. Thus, we may take any positive $\ep$
smaller than
$\ep_{2}$ defined by the following inequality
$\dt \ep_{2}^{-\dt}
\sin(\arccos(M\ep_{2}^{\dt})) - \ep_{2}^{\dt}  m_{2}>0$.

\end{proof}

\no For $k\in \rr $, we denote by $\wf_{k}$ the level set, i.e.
\eqq{\wf_{k}= \{x \in \Omega: \hd \skk(x)=k \}. }{deflev}
The following Corollary describes the structure of level sets in a neighborhood
of $S_{i}$.

\begin{corollary}
For each vertex $S_{i}$, $i=1,\ldots,4$ and for each $k \in \rr$, there is
$\ep_{3}>0$ such that
\m{$\wf_{k}\cap B(S_{i},\ep_{3})$} is an analytic curve with one endpoint in
$S_{i}$.
Moreover, curve $\wf_{k}$ divides $\Omega \cap B(S_{i}, \ep_{3})$ into two
parts: on one of them $\skk>k$ and on the other $\skk<k$.
\label{wnsi}
\end{corollary}

\begin{proof}
For fixed $k\in \rr$ we consider the set $\wf_{k}$. Then, using
Proposition~\ref{podst} we get $\ep_{3}>0$ such that for $\ep\in(0,\ep_{3})$ the
following conditions
\[
\inf_{\Omega \cap \partial B(S_{i},\ep)} \skk< - k, \hd \hd  k< \sup_{\Omega \cap \partial B(S_{i},\ep)} \skk
\]
hold. Let $M$ and $\ep_{2}$ be given by Lemma~\ref{step6}-\ref{h}. Then we
deduce that for each $\ep \in (0, \min\{\ep_{2},\ep_{3}\})$ the set $\wf_{k}\cap
\partial B(S_{i}, \ep) $ consist of one point. Using implicit function theorem
and Lemma~\ref{step2} we conclude the first claim.

\no If we conduct  the same argument as above, for two different numbers $k$,
then we obtain the remaining part of the claim.
\end{proof}

\begin{lemma}
Let us suppose that $\skk$ is given by Corollary \ref{uwj}.  Then,  the set $\{p
\in \Omega: \nabla \skk
(p) = 0 \}$ is finite.

\label{critical}
\end{lemma}

\begin{proof}
 Indeed, $\skk $ is harmonic in simply connected
domains $\Omega_{\pm}=\Omega \cap \{\pm x>- \ep \}$, hence $\skk$ is a
real part of a holomorphic function $f_{\pm}$ in $\Omega_{\pm}$. Then,
the set $\{z=(x,y)\in \Omega_{\pm}: f_{\pm}'(z)=0  \}$  is isolated in
$\Omega_{\pm}$ and from equality $f'(z)=u_{x}(x,y)- i u_{y}(x,y)$ we
deduce that $\{p \in \Omega: \nabla \skk (p) = 0 \}$ is isolated in
$\Omega$. Suppose that this set is not finite. Then, there would  be a
sequence, $p_{n}\in \Omega$, such that $\nabla \skk (p_{n})=0$ and
necessarily $p_{n}\rightarrow p\in\partial \Omega$.

\no We can extend $f$ (respectively, $\skk$) across  flat parts of the
boundary to get a  holomorphic continuation of $f_{\pm}$
(respectively, harmonic  continuation of $\skk$). In this process we
rule out the possibility that $p\in\partial \Omega\setminus\{S_1, S_2, S_3,
S_4\}$. The proof is finished because from Lemma~\ref{step2} we get $\nabla \skk
\not =0$, in a neighborhood of $S_{i}$.
\end{proof}


Now, we will analyze  zero level sets.

\begin{lemma}
There are analytic curves $L_{k}\subseteq \Omega$ such that   $\wfz=
\bigcup_{k=1}^{N}L_{k}$. Moreover, for each $k$ the endpoints of $L_{k}$ belong
to $\partial \Omega$, i.e. $\partial L_{k} \in \g \cup \gf $.
\label{analytic}
\end{lemma}
\begin{proof}
From Lemma~\ref{critical}, we infer that there are finitely many points
$\{p_{m}\}_{m=1}^{m_{0}}$ such that $\nabla \skk (p_{m})=0$ and $p_{m} \in
\wfz$. Therefore, for any $\ep>0$ on  the set $\wfz \setminus
\bigcup_{m=1}^{m_{0}} B(p_{m},\ep) $ we have $\nabla \skk \not = 0$. Hence,
from implicit function theorem each point of its set belongs to some analytic
curve.

On the other hand, for $\ep$ small enough  $\wfz \cap B(p_{m}, \ep )$ is a
set of analytic curves which is analytically equivalent (see  \cite[Definition
2]{liming}) to $\{t e^{i\varphi_{l} }: \hd t\in (-1, 1), \hd l=1,..., l_{0}
\}$, where $\varphi_{l}= \frac{(l-1)\pi+ \frac{\pi}{2}}{l_{0}} $  \hd (see
\cite[Theorem~1]{liming}). This proves the first part of the claim.

Finally, according to  \cite[Theorem~3]{liming}, each analytic curve can be
uniquely extended to the boundary of the domain.
\end{proof}

\begin{remark}
In the above proof number $l_{0}$ is the order of zero of holomorphic function
$f(z)$ such that $f(p_{m})=0$ and $\re{f}= \skk$ in $B(p_{m}, \ep )$.
\label{zeros}
\end{remark}

\begin{lemma}
Let $L$ be a connected subset of $\wfz$  with two endpoints on  $\partial
\Omega$. Then  at least one of them is vertex $S_{i}$ for an $i$ in
the range $1,\ldots,4.$.
\label{endpoints}
\end{lemma}

\begin{proof}
Let us denote  two endpoints of $L$ by $A,B$, they belong to $\partial \Omega$.
We will show that $A$ or $B$ is a vertex  $S_{i}$. For this purpose we
have to exclude all other possibilities. These are:\\
1) $A,B\in \gf$;\\
2) $A\in \g_{i}$ and $B\in \gf$, $i=1,\ldots,4$;\\
3) $A \in \g_{i}$, $B\in \g_{j}$, $i,j=1,\ldots,4$.

We will study them one by one.

1) Let us suppose $A,B\in \gf$. In this case,  $L$ together with a part of $\gf$
bound a nonempty open subset of $\Omega$ and $\skk$ is equal to zero on its
boundary and is harmonic inside. Hence, $\skk \equiv 0$, which is impossible.

2) Let us assume now, that $i=1$, in the other cases we proceed similarly. We
denote the reflection of $A$ (respectively, $B$, $L$) with respect to $\{x=0 \}$
by $A'$ (respectively, $B'$, $L'$). We have to consider the following
subcases:

a)
$A\not = A'$. Then, $L$, $L'$, the part of $\g_{1}$ connecting $A$ and $A'$ and
the part of $\gf$ connecting $B$ and $B'$ bound a nonempty  open subset of
$\Omega$ where $\skk$ is harmonic. Thus, at least one of its  extremal value is
nonzero and necessarily it is located  on $\g_{1}$, because $\skk $ vanishes on
the other parts of the boundary of this set. However, by Hopf Lemma the
normal derivative is nonzero at  extremal points. This fact contradicts the
definition $\skk$.

b)
$A=A'$ and $B\not = B'$. Then, $L$, $L'$ and  the part of $\gf$, connecting $B$
and $B'$, bound a nonempty  open subset of $\Omega$, where $\skk$ is harmonic
and it vanishes on its boundary, but this is impossible.

c)
$A=A'$, $B=B'$ and $L\not = L'$. Then, $L$, $L'$  bound a nonempty  open subset
of $\Omega$, where $\skk$ is harmonic and it vanishes on its boundary,  this is
again impossible.

d)
$L=L'$, i.e. $L\subseteq \{x=0 \}$. By definition, $\skk$ is symmetric with
respect to  $\{y=0 \}$, thus   $\skk_{|\{x=0 \}}=0$. Then, $\skk_{|\{x>0\}}$ can
by uniquely extended to a harmonic function in $\Omega$ by the odd reflection.
On the other hand, $\skk$ is even with respect to $\{x=0\}$. Therefore,
$\skk_{|\{x<0\}}\equiv 0$,  which  is a contradiction.

3) When  $A \in \g_{i}$, $B\in \g_{j}$, $i,j=1,\ldots,4$, we again consider subcases:

a)
$i=j$. Then $L$ and segment $\overline{AB}\subseteq \g_{i}$ bound a nonempty
open subset  of $\Omega$, where $\skk$ is harmonic. Therefore, at least one of
its extremal value is nonzero and necessarily it is located  on $\g_{i}$. By
Hopf Lemma the normal derivative is nonzero at this extremal point. This
contradicts the definition $\skk$.

b)
$j=i+1$. If we denote the reflection of $L$ ($L'$ resp.)
with respect to $\{y=0 \}$ by $L''$ ($L'''$ resp.), then $L,L', L'', L'''$ bound
a neighborhood  of
vertices $\{ S_{l}, l=1,...,4\}$. Outside of this neighborhood $\skk$ is bounded
and harmonic and at least one of its  extremal value is nonzero and necessary it
is located on some $\g_{l}$, but by Hopf Lemma the normal derivative is nonzero
in the extremal point. This is a contradiction with the definition $\skk$.

c)
$j=i+2$. We argue as above.

After having considered all cases we reached the desired result.
\end{proof}

\begin{lemma}
$\nabla \skk (p)\not = 0$ for all $p\in \wfz$.
\label{nocritical}
\end{lemma}

\begin{proof}
Suppose that $\nabla \skk (p_{0}) = 0$ at $p_{0}\in \wfz$. Then, (see
\cite[Theorem~1]{liming} and  also Remark~\ref{zeros}) $p_{0}$ is bifurcation
point and it belongs to at least two analytic curves $L_{k}$, $L_{k'}$ given by
Lemma~\ref{analytic}. Hence, $\widetilde{L}\equiv L_{k}\cup L_{k'} $ is
connected with four endpoints $\{A,B,C,D \}\subseteq \partial \Omega$. If
$L\subseteq \widetilde{L}$ is connected with two endpoints, then from
Lemma~\ref{endpoints}, we get that at least one of them is in $\{S_{i}: \hd
i=1,...,4 \}$. Thus, we deduce that at most one of the endpoints $\{A,B,C,D \}$
is not in  $\{S_{i}: \hd i=1,...,4 \}$. But then from symmetries, we conclude
that all endpoints belong to $\{S_{i}: \hd i=1,...,4 \}$.

From Corollary~\ref{wnsi} we deduce that  $L_{k}$ and $L_{k'}$ connect two
different pairs of vertices  $\{S_{i}: \hd i=1,...,4 \}$ and $p_{0}\in L_{k}\cap
L_{k'}$. It means that $L_{k}\cup L_{k'}$ bound some neighborhood of $\g$ and
outside of it $\skk$ is harmonic and equal to zero on the boundary. This gives a
contradiction.

\end{proof}

\begin{corollary}
Each analytic curve $L_{k}$ from Lemma~\ref{analytic} has at least one endpoint
in the set $\{S_{i}: \hd i=1,...,4 \}$. Moreover, $\wf_{0}= L_{1}\cup L_{2}$ and
endpoints of $L_{i}$
are in vertices $S_{i}$, $i=1,...4$ or $\wf_{0}= L_{1}\cup L_{2}\cup L_{3}\cup L_{4}$ .
\label{corostr}
\end{corollary}

\begin{proof}
If $p\in \wf_{0}$, then from Lemma~\ref{analytic} $p$ belongs to an analytic
curve $L_{k}$ with endpoints on $\partial \Omega$. Then, using
Lemma~\ref{endpoints}, we deduce that at least one endpoint of $L_{k}$ is
$S_{i}$, for an index $i=1,\ldots,4$. If the other endpoint is in some $S_{j}$
for $i \not = j$, then
$\wf_{0}$ is a sum of two analytic curves. If it is not a case, then
$\wf_{0}$ is a sum of four analytic curves.
\end{proof}

\begin{lemma}
Let us suppose that $\skk$ is given by definition \ref{def2}. We
set $\alpha=\sup\limits_{\gj}{\skk}$ and
$\beta=\inf\limits_{\g_{2}}{\skk}$. Then $\alpha \leq 0$ or $\beta \geq 0$.
\label{lemmain}
\end{lemma}

\begin{proof}
We will analyze the structure of zero level set $\wf_{0}$. Let $L_{1}$ be
the analytic curve given in Corollary~\ref{corostr}, such that one its endpoint
is in $S_{1}$.
Then, the second endpoint of $L_{1}$ may \\
a) be equal $S_{2}$;\quad b) be equal $S_{4}$;\quad c) belong to $\gf$;\quad
d) belong to $\g_{1}$;\quad e) belong to $\g_{4}$.
Other possibilities are eliminated by symmetries of $\skk$ and
Lemma~\ref{nocritical}.

We will show that  $\alpha\leq 0$ or $\beta\geq 0$ in all these cases (a)-(e).

a)
Suppose that $L_{1}$ is a curve connecting $S_{1}$ and   $S_{2}$ (see fig.~$1$).
Then $\wf_{0}=L_{1}\cup L_{2} $, where $L_{2}$ is an analytic curve connecting
$S_{3}$ and $S_{4}$.   We denote an open subset of $\Omega$ which
consists of two regions bounded  by curves
$L_{1}$,  $\g_{1}$ and $L_{2}$,  $\g_{3}$ (we use the symmetries of $\skk$)  by
$U$. We notice that  function $\skk$ should be negative in $U$. Indeed,  because
otherwise function $\skk_{|U}$ would have a positive maximum located on
$\g_{1}$. But this is not permitted by Hopf Lemma, because
$\skk$ satisfies condition (\ref{d})${}_{2}$. Furthermore, we deduce that $\skk$
is positive in $U^{c}\equiv \Omega\setminus \overline{U}$. This is indeed the
case, because $\skk$ is positive in a neighborhood of $S_{1}$, contained in
$U^{c}$ (see Corollary~\ref{wnsi}) and  $\wfz \cap U^{c}=\emptyset$, i.e. $\skk$
can not be negative by intermediate value theorem. Thus $\alpha\leq 0$
\textbf{and} $\beta\geq 0$.

b)
Suppose that $L_{1}$ is a curve connecting $S_{1}$ and   $S_{4}$ (see fig.~$2$).
Then, proceeding analogously, we get $\alpha\leq 0$ \textbf{and} $\beta\geq 0$.

\begin{picture}(200,300)(50,30)
\put(80,300){\line(1,0){150}} \put(80,50){\line(1,0){150}}
\put(80,50){\line(0,1){250}} \put(230,50){\line(0,1){250}}

\put(128,220){\line(1,0){54}} \put(128,130){\line(1,0){54}}
\put(128,130){\line(0,1){90}} \put(182,130){\line(0,1){90}}

\setquadratic \plot 182 220 212 250 155 260 98 250 128 220 /

\setquadratic \plot 182 130 212 100 155 90 98 100 128 130 /

\put(146,240){\scriptsize{$\skk <0$}}
\put(146,105){\scriptsize{$\skk <0$}}

\put(92,175){\scriptsize{$\skk >0$}}
\put(200,175){\scriptsize{$\skk >0$}}

\put(153,210){\scriptsize{$\g_{1}$}}
\put(153,135){\scriptsize{$\g_{3}$}}
\put(117,212){\scriptsize{$S_{2}$}}
\put(184,212){\scriptsize{$S_{1}$}}
\put(117,133){\scriptsize{$S_{3}$}}
\put(184,133){\scriptsize{$S_{4}$}}
\put(160,302){\scriptsize{$\gf$}}
\put(142,35){Fig. 1}
\label{picone}
\end{picture}
\begin{picture}(200,300)(50,30)
\put(80,300){\line(1,0){150}} \put(80,50){\line(1,0){150}}
\put(80,50){\line(0,1){250}} \put(230,50){\line(0,1){250}}

\put(128,220){\line(1,0){54}} \put(128,130){\line(1,0){54}}
\put(128,130){\line(0,1){90}} \put(182,130){\line(0,1){90}}

\setquadratic \plot 182 220 208 245 218 175 208 105 182 130 /

\setquadratic \plot 128 220 102 245 92 175 102 105 128 130 /

\put(146,240){\scriptsize{$\skk <0$}}
\put(146,105){\scriptsize{$\skk <0$}}

\put(99,175){\scriptsize{$\skk >0$}}
\put(190,175){\scriptsize{$\skk >0$}}

\put(153,210){\scriptsize{$\g_{1}$}}
\put(153,135){\scriptsize{$\g_{3}$}}
\put(117,212){\scriptsize{$S_{2}$}}
\put(184,212){\scriptsize{$S_{1}$}}
\put(117,133){\scriptsize{$S_{3}$}}
\put(184,133){\scriptsize{$S_{4}$}}
\put(160,302){\scriptsize{$\gf$}}

\put(142,35){Fig. 2}

\label{pictwo}
\end{picture}

c) If $L_{1}$ connects $S_{1}$ and $\gf$, then $\wf_{0}= L_{1}\cup L_{2} \cup
L_{3} \cup L_{4}$, where $L_{i} $ are analytic curves with one endpoint in
$S_{i}$ and the second one on $\gf$ (see fig.~$3$).  Hence, $\Omega$ is divided
onto four regions. Arguing as earlier, we deduce that in the region above
$\g_{1}$ function $\skk$ is negative, but in the region on the right of $\g_{4}$
function $\skk$ is positive. Thus $\alpha\leq 0$ \textbf{and} $\beta\geq 0$.

d)
If the second endpoint of $L_{1}$ is on $\g_{1}$, then by symmetries of the
problem, we deduce that $\wf_{0}= L_{1}\cup L_{2} \cup L_{3} \cup L_{4}$,
where $L_{i} $, $i=1,\ldots,4$ are analytic curves with one endpoint $S_{i}$
and the second one in $\g_{1}$ or $\g_{3}$ (see fig.~$4$).  Then, we denote an
open subset of $\Omega$, consisting of four  regions bounded  by
curves $L_{i}$, $i=1,\ldots,4$, by
$U$. In set $U^{c}$ function $\skk$ is positive, because there are points with
this property in 
$U^{c}$ 
and $\wfz \cap U^{c}
=\emptyset.$ Thus, in this case we can only show that $\beta \geq0$. Hence
$\alpha\leq 0$ \textbf{or} $\beta\geq 0$.

e)
If the second endpoint of $L_{1}$ is in $\g_{4}$, then proceeding similarly as
above we deduce that $\alpha \leq 0$, hence $\alpha\leq 0$ \textbf{or}
$\beta\geq 0$.

Therefore, in any case $\alpha\leq 0$ \textbf{or} $\beta\geq 0$  and the proof
in finished.

\begin{picture}(200,300)(50,30)
\put(80,300){\line(1,0){150}} \put(80,50){\line(1,0){150}}
\put(80,50){\line(0,1){250}} \put(230,50){\line(0,1){250}}

\put(128,220){\line(1,0){54}} \put(128,130){\line(1,0){54}}
\put(128,130){\line(0,1){90}} \put(182,130){\line(0,1){90}}

\setquadratic \plot 128 220  105 260  90 300 /
\setquadratic \plot 182 220  205 260  220 300 /
\setquadratic \plot 128 130  105 90  90 50 /
\setquadratic \plot 182 130  205 90  220 50 /

\put(146,240){\scriptsize{$\skk <0$}}
\put(146,105){\scriptsize{$\skk <0$}}

\put(99,175){\scriptsize{$\skk >0$}}
\put(190,175){\scriptsize{$\skk >0$}}

\put(153,210){\scriptsize{$\g_{1}$}}
\put(153,135){\scriptsize{$\g_{3}$}}
\put(117,212){\scriptsize{$S_{2}$}}
\put(184,212){\scriptsize{$S_{1}$}}
\put(117,133){\scriptsize{$S_{3}$}}
\put(184,133){\scriptsize{$S_{4}$}}
\put(160,302){\scriptsize{$\gf$}}
\put(142,35){Fig. 3}
\label{picthree}
\end{picture}
\begin{picture}(200,300)(50,30)
\put(80,300){\line(1,0){150}} \put(80,50){\line(1,0){150}}
\put(80,50){\line(0,1){250}} \put(230,50){\line(0,1){250}}

\put(128,220){\line(1,0){54}} \put(128,130){\line(1,0){54}}
\put(128,130){\line(0,1){90}} \put(182,130){\line(0,1){90}}

\setquadratic \plot 128 220 113 237 122 250 142 245 148 220 /
\setquadratic \plot 182 220 197 237 188 250 168 245 162 220 /

\setquadratic \plot 128 130 113 113 122 100 142 105 148 130 /
\setquadratic \plot 182 130 197 113 188 100 168 105 162 130 /

\put(123,229){\scriptsize{$\skk <0$}}
\put(168,229){\scriptsize{$\skk <0$}}

\put(121,109){\scriptsize{$\skk <0$}}
\put(170,109){\scriptsize{$\skk <0$}}

\put(99,175){\scriptsize{$\skk >0$}}
\put(190,175){\scriptsize{$\skk >0$}}

\put(153,210){\scriptsize{$\g_{1}$}}
\put(153,135){\scriptsize{$\g_{3}$}}
\put(117,212){\scriptsize{$S_{2}$}}
\put(184,212){\scriptsize{$S_{1}$}}
\put(117,133){\scriptsize{$S_{3}$}}
\put(184,133){\scriptsize{$S_{4}$}}
\put(160,302){\scriptsize{$\gf$}}
\put(142,35){Fig. 4}
\label{picfour}
\end{picture}

\end{proof}

\begin{proof}[Proof of theorem~\ref{tm2}] Let us denote $\alpha_{1}=
\int_{\g_{1}}\skk $ and $\beta_{1}= \int_{\g_{2}}\skk $. Then, from
Lemma~\ref{lemmain}, we get $\alpha_{1}<0$ or $\beta_{1}>0$ and the claim
follows
Corollary~\ref{wnreg}.

\end{proof}

\subsection{A square inside a square}

The situation is much simpler if we assume that $r_{1}=r_{2}$, i.e. $R_{1}$ and
$R_{2}$ are squares. Then, we can say more about properties of the very weak
solutions $\skk$, because  the domain $\Omega$ enjoys additional symmetry. Here
is our first observation
\begin{proposition}\label{uwd}
 If the rectangle $R_1$ in the definition of $\Omega$ is a square, i.e. $\rj=\rd$,
then $\skk(x,y)=-\skk(y,x)$. In particular, $\skk(x,x)=\skk(-x,x)=\skk(x,-x)=\skk(-x,-x)=0$.
\end{proposition}

\begin{proof}
After rotating function $\skk$ by angle $\frac{\pi}{2}$, i.e. after the change
of variables $(x,y)\mapsto (-y,x)$, we get a function $\skk(-y,x)$
satisfying (\ref{d})-(\ref{f}). Thus, by  Corollary \ref{uwj}, we have
$\skk(-y,x)=\skk(x,y)$ or $\skk(-y,x)=-\skk(x,y)$. More precisely, from
(\ref{e}), we have $\skk(y,x)=\skk(x,y)$ or $\skk(y,x)=-\skk(x,y)$.
If the first possibility held, then from definition of $\skk$
we would get $w(x,y)+\sum\limits_{i=1}^{4}\eta_{i}
\varrho_{i}^{-\frac{2}{3}} \coi=w(y,x)+\sum\limits_{i=1}^{4}\eta_{i}
\varrho_{i}^{-\frac{2}{3}}
\cos{\frac{2}{3}(\frac{3}{2}\pi-\te_{i})}$, which is impossible, because
$\coi \ne \cos(\pi-\frac{2}{3}\te_{i})$.
\end{proof}

\begin{lemma}
Let us suppose that $\skk$ is given by Definition \ref{def2} and   the rectangle
$R_1$ in the definition of $\Omega$ is a square.  We
set $\alpha=\sup\limits_{\gj}{\skk}$ and $\beta=\inf\limits_{\g_{2}}{\skk}$. Then $\alpha <0 $ and $\beta > 0$.
\label{lemmainq}
\end{lemma}

\begin{proof}
From Corollary~\ref{corostr} and from the above proposition we deduce that
$\wf_{0}$
consists only of the four segments, each of them  connects  vertices  $S_{i}$
and
$\widetilde{S}_{i}$, $i=1, ..., 4$ (see fig. $5$).

\vspace{1.4cm}

\begin{picture}(300,300)(0,0)
\put(80,300){\line(1,0){250}} \put(80,50){\line(1,0){250}}
\put(80,50){\line(0,1){250}} \put(330,50){\line(0,1){250}}

\put(160,220){\line(1,0){90}} \put(160,130){\line(1,0){90}}
\put(160,130){\line(0,1){90}} \put(250,130){\line(0,1){90}}

\put(250,220){\line(1,1){80}} \put(160,220){\line(-1,1){80}}
\put(160,130){\line(-1,-1){80}} \put(250,130){\line(1,-1){80}}







\setquadratic \plot 250 220 260 225 261 220 258 215 250 210 /

\put(228,202){\scriptsize{$U^{M}_{4}$}}

\setquadratic \plot 250 220 255 230 250 231 245 228 240 220 /

\put(243,212){\line(2,5){6}} \put(243,208){\line(3,2){13}}

\put(200,210){\scriptsize{$\g_{1}$}}
\put(200,135){\scriptsize{$\g_{3}$}}
\put(200,290){\scriptsize{$\gf_{1}$}}
\put(200,55){\scriptsize{$\gf_{3}$}}
\put(85,170){\scriptsize{$\gf_{2}$}}
\put(165,170){\scriptsize{$\g_{2}$}}
\put(238,170){\scriptsize{$\g_{4}$}}
\put(317,170){\scriptsize{$\gf_{4}$}}

\put(280,180){\scriptsize{$\skk>0$}}
\put(110,180){\scriptsize{$\skk > 0$}}
\put(192,85){\scriptsize{$\skk <0$}}
\put(192,260){\scriptsize{$\skk <0$}}


\put(265,280){\scriptsize{$\skk(x,x)=0$}}
\put(105,280){\scriptsize{$\skk(-x,x)=0$}}

\put(190, 30){Fig. 5.}

\end{picture}

\no Then, arguing as in part c) of the proof of Lemma~\ref{lemmain}, we get
$\alpha
\leq 0 $ and $\beta \geq 0$. If $\alpha=0$, then $\skk(p)=0$ for some $p\in
\g_{1}$ and then $p$ would be the extremal point for $\skk$, restricted to the
subset
of $\Omega$, bounded by $\g_{1}$, $\gf_{1}$ and segments $(\pm x,  x)$, $x\in
(r_{1}, \la_{0} r_{1})$. Therefore, by Hopf Lemma $\frac{\partial \skk}{\partial
n }(p)>0$, which contradicts (\ref{d})${}_{2}$, hence $\alpha<0$.
Finally, from Proposition~\ref{uwd}, we get $\beta= - \alpha>0$.
\end{proof}

\begin{proof}[Proof of theorem~\ref{tm1}] Let us denote $\alpha_{1}=
\int_{\g_{1}}\skk $ and $\beta_{1}= \int_{\g_{2}}\skk $. Then from
Lemma~\ref{lemmainq} and Proposition~\ref{uwd} we get $\alpha_{1}=-\beta_{1}<0$
and the claim follows Corollary~\ref{wnreg}.
\end{proof}





\subsection*{Acknowledgment}
This work has been supported by the European Union in the framework of European Social Fund through the Warsaw University of Technology Development Programme, realized by Center for Advanced Studies. Both authors were in part supported by NCN through 2011/01/B/ST1/01197
grant.


\begin{thebibliography}{99}
\addcontentsline{toc}{section}{Bibliography}
%
\bibitem{berg}  W.F.Berg,  Crystal growth from solutions.
{\it Proc. Roy. Soc. London A}, {\bf 164},  79--95 (1938).
%
\bibitem{dauge}
M.Dauge, {\it Elliptic boundary value problems on corner domains. Smoothness
and asymptotics of solutions},  Lecture Notes in Mathematics, 1341.
Springer-Verlag, Berlin, 1988.
%
\bibitem{gr-berg} Y.Giga,  P.Rybka,  Berg's Effect,
{\it Adv. Math. Sci. Appl.,}  {\bf 13}, (2003), 625--637.
%
\bibitem{gk} P.G\'orka, A.Kubica, private communication, 2004.
%
\bibitem{Grisvard}  P. Grisvard, \textit{Elliptic problems in nonsmooth
domains}, Pitman, London,  1985.
%
\bibitem{kondratiew} V.A.Kondra$\rm t^\prime\!$ev,  Boundary value problems
for elliptic equations
in domains with conical or angular points.
Trudy Moskov. Mat. Ob\v s\v c.  {\bf 16}, 209--292 (1967).
%
\bibitem{kozlov}
V.A.Kozlov, V.G.Maz${}'$ya, J.Rossmann,
\textit{Spectral problems associated with corner singularities of solutions to
elliptic equations},
Mathematical Surveys and Monographs, 85. American Mathematical Society,
Providence, RI, 2001.
%
\bibitem{K1}
 A. Kubica, The regularity of weak and
very weak solutions of the Poisson equation on polygonal domain with
mixed boundary conditions (part I), {\it Appl. Math. (Warsaw),} {\bf 31}
(2004), no. 1, 443--456.


%
\bibitem{K}
 A. Kubica, The regularity of weak and
very weak solutions of the Poisson equation on polygonal domain with
mixed boundary conditions (part II), \textit{Appl. Math. (Warsaw)}, {\bf 32}
(2005), no. 1, 17-36.
%
\bibitem{liming}
Wen, Zhi Ying; Wu, Li Ming; Zhang, Yiping
Set of zeros of harmonic functions of two variables. Harmonic analysis (Tianjin,
1988), 196--203,
Lecture Notes in Math., 1494, Springer, Berlin, 1991.
%
%
\bibitem{nazarov} S.A.Nazarov,  B.A.Plamenevsky, {\it Elliptic  problems in
domains with
piecewise smooth boundaries}. Berlin, Walter de Gruyter, 1994.
%
\bibitem{Sard}
A.Sard, The measure of the critical values of differentiable maps, {\it  Bull.
Amer. Math. Soc.}, {\bf 48},  (1942), 883--890.
%
\bibitem{seeger} A.Seeger,  Diffusion problems associated with the growth of
crystals
from dilute solution. \textit{Philos. Mag.}, ser. 7 {\bf 44}, no 348, 1--13
(1953).


\end{thebibliography}
\end{document}